\documentclass{amsart}

\usepackage{amsmath,amsfonts,amssymb,color,hyperref, enumerate}

\usepackage{amsmath, amsfonts,amsthm,amssymb,amscd, verbatim,graphicx,color,multirow,tikz,tikz-cd,booktabs, caption, mathdots,bm,chngcntr,adjustbox,longtable
}
\usepackage{tikz-cd}
\usetikzlibrary{positioning}
\newtheorem{theorem}{Theorem}[section]
\newtheorem{lemma}[theorem]{Lemma}
\newtheorem{proposition}[theorem]{Proposition}

\newtheorem{example}[theorem]{Example}

\begin{document}

\title[Normal covers of diagonal type]{Groups having minimal covering number 2 of diagonal type}

\author[M. Fusari]{Marco Fusari}
\address{Marco Fusari, Dipartimento di Matematica ``Felice Casorati", University of Pavia, Via Ferrata 5, 27100 Pavia, Italy}
\email{lucamarcofusari@gmail.com}

\author[A.~Previtali]{Andrea Previtali}
\address{Dipartimento di Matematica e Applicazioni, University of Milano-Bicocca, Via Cozzi 55, 20125 Milano, Italy}
\email{andrea.previtali@unimib.it}

\author[P.~Spiga]{Pablo Spiga}
\address{Dipartimento di Matematica e Applicazioni, University of Milano-Bicocca, Via Cozzi 55, 20125 Milano, Italy}
\email{pablo.spiga@unimib.it}

\begin{abstract}
Garonzi and Lucchini~\cite{GL} explored finite groups $G$ possessing a normal $2$-covering, where no proper quotient of $G$ exhibits such a covering. Their investigation offered a comprehensive overview of these groups, delineating that such groups fall into distinct categories: almost simple, affine, product action, or diagonal.

In this paper, we focus on the family falling under the diagonal type. Specifically, we present a thorough classification of finite diagonal groups  possessing a normal $2$-covering, with the attribute that no proper quotient of $G$ has such a  covering.

\textit{With deep appreciation to Martino Garonzi and Andrea Lucchini, for keeping us entertained.}

\end{abstract}


\maketitle
\section{Introduction}\label{sec:intro}

Let $G$ be a finite group. A \textit{\textbf{normal $k$-covering}} of $G$ is a family of $k$ proper subgroups $H_1,\ldots,H_k$ of $G$ such that each element of $G$ is conjugate to an element of $H_i$ for some $i$, that is, $$G=\bigcup_{i=1}^k\bigcup_{g\in G}H_i^g.$$ The minimal $k\in\mathbb{N}$ such that $G$ admits a normal $k$-covering is said to be the \textit{\textbf{normal covering number}} of $G$. Jordan's theorem ensures that $G$ cannot be covered by the conjugates of a proper subgroup and hence $\gamma(G)\ge 2$.

The study of finite groups admitting a normal $2$-covering has attracted considerable attention and we refer to the introductory chapter of the monograph~\cite{1} for details. For inductive purposes, it is interesting to classify finite groups $G$ with $\gamma(G)=2$ such that $\gamma(G/N)>2$, for every non-identity normal subgroup of $G$. In this paper, a group with this property is called \textit{\textbf{basic}}. Broadly speaking,  basic groups are the basic objects having normal covering number $2$, because any other such group is an extension of these building blocks.

Basic groups with a normal covering number of $2$ are explored by Garonzi and Lucchini in~\cite{GL}, revealing that such groups manifest as either almost simple, affine, product action, or diagonal. While we provide a concise overview of these classifications here, a comprehensive understanding is available in~\cite{GL}, where the terminology draws inspiration from the O'Nan-Scott classification of primitive groups. Indeed, let $G$ be a basic group with $\gamma(G)=2$ and let $H $ and $K$ be proper subgroups of $G$ witnessing that $\gamma(G)=2$. Replacing $H$ and $K$ if necessary, we may suppose that $H$ and $K$ are maximal subgroups of $G$. Garonzi and Lucchini show that $G$ has a unique minimal normal subgroup, $N$ say. Moreover, they show that either $G$ is almost simple or, replacing $H$ with $K$ if necessary, $N\le H$. Almost simple groups admitting a normal $2$-covering where neither $H$ nor $K$ contains the socle of $G$ are completely determined in the monograph~\cite{1}. Therefore, here we suppose $N\le H$.
Since $\gamma(G/N)>2$, we deduce $N\nleq K$ and hence $K$ is a maximal core-free subgroup of $G$. In particular, the permutation action of $G$ on the cosets of $K$ gives rise to a faithful representation of $G$ as a primitive permutation group. Under this permutation representation, $K$ can be viewed as a point stabilizer. Using the terminology from the O'Nan-Scott classification of primitive groups, Garonzi and Lucchini show that $G$ is either almost simple, affine, product action or diagonal. In particular, in the context of basic normal $2$-coverings these terms are referred to the abstract structure of $G$ and to the embedding of $K$ in $G$.

The scope of this paper is to give a complete classification of the basic groups falling into the diagonal family. We refer to~\cite{0} for an impressive source of information on the algebraic, combinatorial and geometric structure of groups of diagonal type. In Section~\ref{sec:diagonal type}, we have included a definition of this class of groups and a number of basic facts, which are sufficient for the purpose of this paper.

\begin{example}\label{example}{\rm
Let $T$ be a non-abelian simple group, let $U$ be a subgroup of $\mathrm{Aut}(T)$ containing the inner automorphisms of $T$ and let $p$ be a prime number with $\gcd(|U|,p)=1$. Here we identify $T$ with the inner automorphism group of $T$. Let $N=T^p$, let 
$$H=\{(x_1,\ldots,x_p)\in U^p\mid x_i\equiv x_j\pmod T, \forall i,j\in \{1,\ldots,p\}\}$$ and let $\sigma=(1\,2\,\cdots\,p)$ be the cyclic permutation of degree $p$. 

We define 
$$G=H\rtimes \langle \sigma\rangle,$$
where $\sigma$ acts as a group of automorphisms on $H$ by setting
$$(x_1,x_2,\ldots,x_p)^\sigma=(x_2,x_3,\ldots,x_p,x_1),$$
for each $(x_1,\ldots,x_\ell)\in H$.

Finally we let $$K=\{(x,\ldots,x)\in H\mid x\in U\}\times \langle\sigma\rangle\le G.$$ Observe that, as $\gcd(|U|,p)=1$, a Sylow $p$-subgroup of $G$ has order $p$ and hence $\langle\sigma\rangle$ is a Sylow $p$-subgroup of $G$. Moreover, ${\bf C}_G(\sigma)=K$.

We claim that $H$ and $K$ are the components of a normal $2$-covering of $G$. Indeed, let $g\in G$. Therefore, $g=(x_1,\ldots,x_p)\sigma^i$, for some $(x_1,\ldots,x_p)\in H$ and $i\in \{0,\ldots,p-1\}$. If $i=0$, then $g\in H$. Assume then $i\ne 0$.  As $i\ne 0$, $g$ has order divisible by $p$ and hence $g$ centralizes a Sylow $p$-subgroup $P$ of $G$. By Sylow's theorem, there exists $z\in G$ with $P^z=\langle \sigma\rangle$. Thus $g^z$ centralizes $P^z=\langle \sigma\rangle$ and hence $g^z\in {\bf C}_G(\sigma)=K$.}
\end{example}

We are now ready to state our main result.
\begin{theorem}\label{thrm:1}
Let $G$ be a group with $\gamma(G)=2$ and $\gamma(G/N)>2$, for every non-identity normal subgroup $N$ of $G$. Let $H$ and $K$ be maximal subgroups of $G$ witnessing that $\gamma(G)=2$; moreover, in view of~\cite{GL}, let $H$ be the component containing the socle of $G$. If $G$ is of diagonal type, then $G,K$ and $H$ are isomorphic to one of the groups described in Example~$\ref{example}$.
\end{theorem}

Not every group mentioned in Example~\ref{example} qualifies as basic. For instance if $U\cong\mathrm{P}\Omega_8^+(q).\mathrm{Sym}(3)$, then $G$ is not basic, because $\gamma(\mathrm{Sym}(3))=2$ and $\mathrm{Sym}(3)$ is an epimorphic image of $G$.
The basic nature of $G$ hinges on the structure of $G/N = U/T \times \langle\sigma\rangle$, and consequently, on the outer automorphism group of $T$. While we refrain from delving into every intricate detail needed for an ``if and only if'' scenario in Theorem~\ref{thrm:1}, it is worth noting that the outer automorphism group of a non-abelian simple group is reasonably straightforward, thereby rendering the structure of $U/T\times\langle\sigma\rangle$ equally uncomplicated.

\section{Groups of diagonal type}\label{sec:diagonal type}
We start by recalling the structure of the finite primitive groups of diagonal type. This will also allow us to set up the
notation for this section and for the proof of Theorem~\ref{thrm:1}.

Let $\ell \ge 1$ and let $T$ be a non-abelian simple group. Consider the group $N = T^{\ell+1}$ and $D = \{(t, \ldots , t) \in N \mid t \in T \}$, a
diagonal subgroup of $N$. Set $\Omega = N/D$, the set of right cosets of $D$ in $N$. Then $|\Omega| = |T |^\ell$. Moreover we may identify
each element $\omega \in \Omega$ with an element of $T^\ell$ as follows: the right coset $\omega = D(\alpha_0 , \alpha_1 , \ldots , \alpha_\ell )$ contains a unique element whose first coordinate is $1$, namely, the element $(1, \alpha_0^{-1}\alpha_1,\ldots, \alpha_0^{-1}\alpha_\ell )$. We choose this distinguished coset representative.
Now the element $\phi$ of $\mathrm{Aut}(T )$ acts on $\Omega$ by
$$D(1, \alpha_1, \ldots, \alpha_\ell )^\phi = D(1, \alpha_1^\phi,\ldots,\alpha_\ell^\phi).$$
Note that this action is well-defined because $D$ is $\mathrm{Aut}(T )$-invariant. Next, the element $(t_0 ,\ldots, t_\ell )$ of $N$ acts on $\Omega$ by
$$D(1,\alpha_1, \ldots,\alpha_\ell )^{(t_0,\ldots,t_\ell)} = D(t_0 , \alpha_1 t_1 ,\ldots, \alpha_\ell t_\ell ) = 
D(1, t_0^{-1}\alpha_1 t_1 , \ldots, t_0^{-1} \alpha_\ell t_\ell ).$$
Observe that the action induced by $(t, \ldots , t) \in N$ on $\Omega$ is the same as the action induced by the inner automorphism
corresponding to the conjugation by $t$. Indeed,
\begin{align*}
D(1,\alpha_1,\ldots,\alpha_\ell)^{(t,\ldots,t)}&=D(t,\alpha_1t,\ldots,\alpha_\ell t)=D(t^{-1},t^{-1},\ldots,t^{-1})(t,\alpha_1t,\ldots,\alpha_\ell t)\\
&=D(1,\alpha_1^{t},\ldots,\alpha_\ell^t)=D(1,\alpha_1,\ldots,\alpha_\ell)^{\phi_t},
\end{align*}
where we are denoting by $\phi_t:T\to T$ the inner automorphism determined by $t$.
Finally, the element $\sigma$ in $\mathrm{Sym}(\{0,\ldots, \ell\})$ acts on $\Omega$ simply by permuting the
coordinates. Note that this action is well-defined because $D$ is $\mathrm{Sym}(\ell + 1)$-invariant.

The set of all permutations we described generates a group $W$ isomorphic to
$$T^{\ell+1} \cdot (\mathrm{Out}(T ) \times \mathrm{Sym}(\ell + 1)).$$
In particular, each element $g$ of $W$ can be written in the form $g=n\phi\sigma$, for some $n\in N$, $\phi\in\mathrm{Aut}(T)$ and $\sigma\in\mathrm{Sym}(\ell+1)$. Now, $n=(t_0,t_1,\ldots,t_\ell)$, where $t_0,\ldots,t_\ell\in T$. Since $N\cap\mathrm{Aut}(T)=D$, without loss of generality, we may choose $n\in N$ with $t_0=1$. In particular, each element $g$ of $W$ can be written uniquely in the form $g=n\phi\sigma$, for some $n=(t_0,t_1,\ldots,t_\ell)\in N$, $\phi\in\mathrm{Aut}(T)$ and $\sigma\in\mathrm{Sym}(\ell+1)$, where the first coordinate $t_0$ of $n$ equals $1$.

A subgroup $G$ of $W$ containing the socle $N$ of $W$ is primitive if either $\ell = 1$ or $G$ acts primitively by conjugation on the
$\ell + 1$ simple direct factors of $N$, see~\cite[Theorem 4.5A]{6}. The group $G$ is said to be primitive of diagonal type, when the second
case occurs, that is, $N \unlhd G \le W$ and $G$ acts primitively by conjugation on the $\ell + 1$ simple direct factors of $N$.

We set $\omega_o=D\in\Omega$. In particular, 
\begin{align*}
W_{\omega_o}&=\mathrm{Aut}(T)\times\mathrm{Sym}(\ell+1),\\
N_{\omega_o}&=\mathrm{Inn}(T)=T,
\end{align*}
where we are denoting by $\mathrm{Inn}(T)$ the set of inner automorphisms of $T$ and we are identifying $\mathrm{Inn}(T)$ with $T$ itself,
and $$\frac{W}{N}\cong\frac{W_{\omega_o}}{N_{\omega_o}}\cong \frac{\mathrm{Aut}(T)}{\mathrm{Inn}(T)}\times\mathrm{Sym}(\ell+1)=\mathrm{Out}(T)\times\mathrm{Sym}(\ell+1).$$
We let $$\pi_1:W\to \mathrm{Out}(T)\hbox{ and }\pi_2:W\to\mathrm{Sym}(\ell+1)$$
be the natural projections of $W$ onto the first and the second direct factors of $W/N$. Therefore,
\begin{align*}
\mathrm{Ker}(\pi_1)&=N\rtimes\mathrm{Sym}(\ell+1)\,\hbox{ and }\,\mathrm{Ker}(\pi_2)=N.\mathrm{Out}(T).
\end{align*}

\section{Proof of Theorem~$\ref{thrm:1}$}\label{sec:proof}
We need three preliminary facts.
\begin{proposition}\label{p:1}
Let $L$ be a finite primitive group of degree $n$ containing a cyclic regular subgroup $C$. Then every element of $L$ which is a cycle of length $n$ is $L$-conjugate to an element of $C$.
\end{proposition}
\begin{proof}
We use~\cite{Jones} for the classification of the finite primitive permutation groups containing a cyclic regular subgroup. From this classification we deduce that one of the following holds
\begin{enumerate}
\item\label{item1}$L\le\mathrm{AGL}_1(p)$ with $n=p$ prime, or
\item\label{item2}$L=\mathrm{Sym}(4)$ with $n=4$, or
\item\label{item3}$L=\mathrm{Sym}(n)$ for some $n\ge 5$ or $L=\mathrm{Alt}(n)$ for some odd $n\ge 5$, or
\item\label{item4}$\mathrm{PSL}_d(q)\le L\le \mathrm{P}\Gamma\mathrm{L}_d(q)$, acting on $n=(q^d-1)/(q-1)$ points or hyperplanes,
\item\label{item5}$L=\mathrm{PSL}_2(11)$, $M_{11}$ or $M_{23}$ with $n=11$, $11$ or $23$ respectively.
\end{enumerate}

When $n$ is prime, the result follows from Sylow's theorem and hence the result holds in case~\eqref{item1} and~\eqref{item5}. The result follows with a direct inspection when $L$ is as in~\eqref{item2} or as in~\eqref{item3}. The result follows from~\cite[Corollary~2]{Jones} when $L$ is as in~\eqref{item4}.
\end{proof}
\begin{lemma}\label{l:1}
Let $T$ be a non-abelian simple group and let $\phi\in\mathrm{Aut}(T)$. Then the mapping $T\to T$ defined by $y\mapsto y^{-1}y^\phi$ is not bijective.
\end{lemma}
\begin{proof}
The preimage, under the mapping $y\mapsto y^{-1}y^\phi$, of the identity of $T$ is ${\bf C}_T(\phi)$. Since $T$ is a non-abelian simple group, we have ${\bf C}_T(\phi)\ne 1$, because a non-abelian simple group does not admit a fixed-point-free automorphism by~\cite{Rowley}.
\end{proof}

\begin{lemma}\label{l:new}Let $T$ be a non-abelian simple group, let $p$ be a prime number with $\gcd(p,|T|)=1$ and let $\phi\in\mathrm{Aut}(T)$ be an element having $p$-power order. If ${\bf C}_T(\phi)={\bf C}_T(\phi^p)$, then $\phi=1$.
\end{lemma}
\begin{proof}
It is known that, if there exists a prime divisor $p$ of $| \mathrm{Aut}( T )|$ such that $p$ does not divide $|T |$, then $T$ is
a simple group of Lie type and every $\phi \in\mathrm{Aut}( T )$ of $p$-power order is conjugate to a field automorphism of $T$, see~\cite[Table~5, page~xvii]{atlas}. 

Arguing by contradiction, suppose that $\phi\ne 1$ and let $p^\kappa$ be the order of $\phi$.
So, $T$ is isomorphic to ${}^d L_n (q^f )$, where ${}^dL_n (q^f )$ denotes the group of Lie type $L$ of rank $n$ (untwisted if
$d = 1$, twisted if $d = 2$, and ${}^3 D_4 (q^f )$ if $d = 3$), $q$ being a suitable prime, and $f$ a multiple of $p^\kappa$.

Now, ${\bf C}_T(\phi)={}^dL_n(p^{f/p^\kappa})$ and ${\bf C}_T(\phi^p)={}^dL_n(p^{f/p^{\kappa-1}})$. Therefore, ${\bf C}_T(\phi)<{\bf C}_T(\phi^p)$.
\end{proof}

\begin{proof}[Proof of Theorem~$\ref{thrm:1}$]
Now, let $G$ be as in Section~\ref{sec:diagonal type}, that is, $G\le W$ and $G$ is a primitive group of diagonal type. We suppose that $\gamma(G)=2$ and $\gamma(G/X)>2$ for every non-identity proper normal subgroup $X$ of $G$. 

Since $\gamma(G)=2$, there exists two proper subgroups $H$ and $K$ of $G$ with
\begin{align}\label{gamma=2}
G=\bigcup_{g\in G}H^g\cup\bigcup_{g\in G}K^g.
\end{align}
Replacing $H$ and $K$ if necessary, we may suppose that $H$ and $K$ are maximal subgroups of $G$. From~\cite[Theorem~5 and Section~3]{GL}, replacing $H$ with $K$ if necessary, we suppose that $N\le H$ and $K=G_{\omega_o}$. In particular, 
$$T=\mathrm{Inn}(T)\unlhd K=G_{\omega_o}\le W_{\omega_o}=\mathrm{Aut}(T)\times\mathrm{Sym}(\ell+1).$$

We let $L=\pi_2(G)$ be the image of $G$ under $\pi_2$, that is, $L$ is the primitive permutation group induced by $G$ by its action by conjugation on the $\ell+1$ simple direct factors of $N$. Thus $L=\pi_2(G)=\pi_2(K)$, because $G=NK$. 

We claim
\begin{equation}\label{l:2}\mathrm{Ker}(\pi_2)\cap G\subseteq \bigcup_{x\in G}H^x.
\end{equation}
We argue by contradiction and we suppose that there exists $$g\in (\mathrm{Ker}(\pi_2)\cap G)\setminus \bigcup_{x\in G}H^x.$$ In particular, $g=n\phi$, for some $n\in N$ and some $\phi\in\mathrm{Aut}(T)$. Since $N\le H\cap\mathrm{Ker}(\pi_2)$, replacing $g$ with $\phi$ if necessary, we may suppose that $g=\phi$. 

For each $t\in T$, let $$g_t=(t,\underbrace{1,\ldots,1}_{\ell\textrm{ times}})\phi\in G.$$ From~\eqref{gamma=2}, $g_t$ has a $G$-conjugate in $H$ or a $G$-conjugate in $K$. If $g_t$ has a $G$-conjugate in $H$, then so does $g$, contradicting our assumption on $g$. Therefore, $g_t$ has a $G$-conjugate in $K$. Since $G=NK$, we deduce that $g_t$ has an $N$-conjugate in $K$. Therefore, there exists $(1,t_1,\ldots,t_\ell)\in N$ such that $g_t^{(1,t_1,\ldots,t_\ell)}\in K$. We have
\begin{align*}
g_t^{(1,t_1,\ldots,t_\ell)}&=
(t,t_1^{-1},\ldots,t_\ell^{-1})\phi(1,t_1,\ldots,t_\ell)=
(t,t_1^{-1},\ldots,t_\ell^{-1})(1,t_1^{\phi^{-1}},\ldots,t_\ell^{\phi^{-1}})\phi\\
&=(t,t_1^{-1}t_1^{\phi^{-1}},\ldots,t_\ell^{-1}t_\ell^{\phi^{-1}})\phi.
\end{align*} 
Now, if this element lies in $K$, then $(t,t_1^{-1}t_1^{\phi^{-1}},\ldots,t_\ell^{-1}t_\ell^{\phi^{-1}})\in D$. In particular, $t=t_1^{-1}t_1^{\phi^{-1}}$. Since this whole argument does not depend on $t\in T$, we have shown that the mapping $T\to T$ defined by $y\mapsto y^{-1}y^{\phi^{-1}}$ is bijective. However, this contradicts Lemma~\ref{l:1}. This contradiction has established the veracity of~\eqref{l:2}.

We claim that
\begin{equation}\label{l:22}\mathrm{Ker}(\pi_2)\cap G\leq H.
\end{equation}
We argue by contradiction and we suppose that $\mathrm{Ker}(\pi_2)\cap G\nleq H$. As $H$ is maximal in $G$ and as $\mathrm{Ker}(\pi_2)\cap G\unlhd G$, we get $G=(\mathrm{Ker}(\pi_2)\cap G)H.$
Now, by intersecing both sides of~\eqref{l:2} with $\mathrm{Ker}(\pi_2)\cap G$, we deduce
\begin{align*}\mathrm{Ker}(\pi_2)\cap G&=\mathrm{Ker}(\pi_2)\cap G\cap \bigcup_{x\in G}H^x=\mathrm{Ker}(\pi_2)\cap G\cap\bigcup_{x\in \mathrm{Ker}(\pi_2)\cap G}H^x\\
&=\bigcup_{x\in \mathrm{Ker}(\pi_2)\cap G}(\mathrm{Ker}(\pi_2)\cap H)^x.
\end{align*}
As $\mathrm{Ker}(\pi_2)\cap G\nleq H$, we get that $\mathrm{Ker}(\pi_2)\cap H$ is a proper subgroup of $\mathrm{Ker}(\pi_2)\cap G$. This shows that each element of $\mathrm{Ker}(\pi_2)\cap G$ is $(\mathrm{Ker}(\pi_2)\cap G)$-conjugate to an element of its proper subgroup $\mathrm{Ker}(\pi_2)\cap H$. However, this contradicts Jordan's theorem. This contradiction has established the veracity of~\eqref{l:22}.

Let $A=\pi_2(H)$ be the image of $H$ under $\pi_2$. Since $H<G$ and $\mathrm{Ker}(\pi_2)\cap G\le H$ by~\eqref{l:22}, we deduce 
\begin{equation}\label{cl:1}
A< L.
\end{equation}

Now, let 
\begin{equation}\label{eq:2}\sigma\in L\setminus\bigcup_{\tau\in L}A^\tau.\end{equation}
Since $L=\pi_2(G)=\pi_2(K)$, there exists $k\in K$ with $\sigma=\pi_2(k)$. As $K\le\mathrm{Aut}(T)\times L$, there exists $\phi\in\mathrm{Aut}(T)$ with $k=\phi\sigma\in K$. 

For each $n\in N$, $n\phi\sigma\in G$ and since $\pi_2(n\phi\sigma)=\sigma$ is not $L$-conjugate to an element of $A$, we deduce that $n\phi\sigma$ is not $G$-conjugate to an element of $H$. Therefore, from~\eqref{gamma=2}, $n\phi\sigma$ is $G$-conjugate to an element of $K$. Moreover, as $G=NK$, we deduce that there exists $n'\in N$ with $(n\phi \sigma)^{n'}\in K$. As we have mentioned above, we may choose $n'$ such that its first coordinate is the identity of $T$ and hence $n'=(1,t_1,\ldots,t_\ell)$, for some $t_1,\ldots,t_\ell\in T$. In what follows, we apply this argument to various choices of $n$.

As $\sigma\ne 1$, when written as a product of cycles having disjoint supports, $\sigma$ has a cycle of length $a\ge 2$. Relabeling the indexed set $\{0,\ldots,\ell\}$ if necessary, we may suppose that $(0\,1\ldots\,a-1)$ is a cycle of $\sigma$ of length $a$.

Let $t\in T$, let $n=(t,1,\ldots,1)$ and let $n'=(1,t_1,\ldots,t_\ell)\in N$ with $(n\phi\sigma)^{n'}\in K$. We have
\begin{align}\label{eq:22}
(n\phi \sigma)^{n'}&=(t,t_1^{-1},\ldots,t_{a-2}^{-1},t_{a-1}^{-1},\ldots)\phi\sigma n'\\\nonumber
&=
(t,t_1^{-1},\ldots,t_{a-2}^{-1},t_{a-1}^{-1},\ldots)\phi(t_1,t_2,\ldots,t_{a-1},1,\ldots)\sigma\\\nonumber
&=(tt_1^{\phi^{-1}},t_1^{-1}t_2^{\phi^{-1}},\ldots,t_{a-2}^{-1}t_{a-1}^{\phi^{-1}},t_{a-1}^{-1},\ldots)\phi\sigma.\nonumber
\end{align}
If this element belongs to $K$, then $$tt_1^{\phi^{-1}}=t_1^{-1}t_2^{\phi^{-1}}=\cdots=t_{a-2}^{-1}t_{a-1}^{\phi^{-1}}=t_{a-1}^{-1}.$$
This implies
\begin{align}\label{eq:13}\nonumber
t_{a-1-i}&=t_{a-1}^{\phi^{-i}}t_{a-1}^{\phi^{-(i-1)}}\cdots t_{a-1}^{\phi^{-1}}t_{a-1}, \quad\forall i\in \{1,\ldots,a-2\},\\
t&=t_{a-1}^{-1}(t_{a-1}^{-1})^{\phi^{-1}}\cdots (t_{a-1}^{-1})^{\phi^{-(a-1)}}=(t_{a-1}^{-1}\phi)^{a}\phi^{-a}.
\end{align}

Summing up, so far, we have shown the following fact.
\begin{lemma}\label{l:13}Let $\sigma\in L\setminus\bigcup_{\tau\in L}A^\tau$ and let $\phi\in\mathrm{Aut}(T)$ with $\phi \sigma\in K$. If $\sigma$ has a cycle of length $a\ge 2$, then  the mapping $y\mapsto (y\phi)^a\phi^{-a}$ is a bijection.
\end{lemma}
\begin{proof}
From~\eqref{eq:13}, for each $t\in T$, there exists $t_{a-1}\in T$ such that $t=(t_{a-1}^{-1}\phi)^a\phi^{-a}$. Therefore, the mapping $T\to T$ defined by $y\mapsto (y\phi)^a\phi^{-a}$ is surjective and hence bijective.
\end{proof}

Suppose now that $\sigma$ fixes some point of $\{0,\ldots,\ell\}$. As above, relabeling the indexed set if necessary, we may suppose that $\sigma$ fixes the point $a$. Using this information in~\eqref{eq:22}, we get
$$(n\phi \sigma)^{n'}=(tt_1^{\phi^{-1}},t_1^{-1}t_2^{\phi^{-1}},\ldots,t_{a-2}^{-1}t_{a-1}^{\phi^{-1}},t_{a-1}^{-1},t_a^{-1}t_a^{\phi^{-1}},\ldots)\phi\sigma$$
If this element belongs to $K$, then the elements appearing in the $a^{\mathrm{th}}$ and in the $(a+1)^{\mathrm{th}}$ coordinates are equal, that is, 
$t_{a-1}^{-1}=t_{a}^{-1}t_a^{\phi^{-1}}.$ Substituting this value of $t_{a-1}^{-1}$ in~\eqref{eq:13}, we get
\begin{equation*}
t=(t_a^{-1}t_a^{\phi^{-1}}\phi)^{a-1}\phi^{-(a-1)}=(t_a^{-1}\phi t_a)^{a}\phi^{-a}=t_a^{-1}\phi^{a}t_a\phi^{-a}=t_a^{-1}t_a^{\phi^{-a}}.
\end{equation*}
Since this holds for every $t\in T$, this implies that the mapping $T\to T$ defined by $y\mapsto y^{-1}y^{\phi^a}$ is bijective. However, this contradicts Lemma~\ref{l:1}. This contradiction has shown that, each element of $L$ fixing some point is $L$-conjugate to an element of $A$, that is, the elements in $L\setminus\bigcup_{\tau\in L}A^\tau$ act fixed-point-freely on $\{0,\ldots,\ell\}$.

Now, suppose that our element $\sigma$ that we have chosen in~\eqref{eq:2} has at least two cycles of length at least $2$. As usual relabeling the index set $\{0,\ldots,\ell\}$ we may suppose that $(0\,1\,\ldots\,a-1)$ and $(a\,a+1\ldots a+b-1)$ are two cycles of $\sigma$, having lengths $a\ge 2$ and $b\ge 2$ respectively.  Using this information in~\eqref{eq:22}, we get
\begin{align*}
(n\phi \sigma)^{n'}&=(tt_1^{\phi^{-1}},t_1^{-1}t_2^{\phi^{-1}},\ldots,t_{a-2}^{-1}t_{a-1}^{\phi^{-1}},t_{a-1}^{-1},t_a^{-1}t_{a+1}^{\phi^{-1}},t_{a+1}^{-1}t_{a+2}^{\phi^{-1}},\ldots,t_{a+b-1}^{-1}t_a^{\phi^{-1}},\ldots)\phi\sigma.
\end{align*}
If this element belongs to $K$, then 
$$t_{a-1}^{-1}=t_{a+b-1}^{-1}t_{a}^{\phi^{-1}}\hbox{ and }t_a^{-1}t_{a+1}^{\phi^{-1}}=t_{a+1}^{-1}t_{a+2}^{\phi^{-1}}=\cdots=t_{a+b-1}^{-1}t_a^{\phi^{-1}}.$$
This implies
\begin{align}\label{eq:saturday}
t_{a+b-1-i}&=(t_{a+b-1}t_a^{-1})^{\phi^{-i}}(t_{a+b-1}t_a^{-1})^{\phi^{-(i-1)}}\cdots (t_{a+b-1}t_a^{-1})^{\phi^{-1}}(t_{a+b-1}t_a^{-1})t_a,
\end{align} for every $i\in \{1,\ldots,b-2\}$, and
\begin{align*}
1&=(t_at_{a+b-1}^{-1})(t_{a}t_{a+b-1}^{-1})^{\phi^{-1}}\cdots (t_at_{a+b-1}^{-1})^{\phi^{-(b-1)}}=(t_{a}t_{a+b-1}^{-1}\phi)^{b}\phi^{-b}.
\end{align*}
From Lemma~\ref{l:13}, the mapping $T\to T$ defined by $y\to (y\phi)^b\phi^{-b}$ is bijective. Since $1\mapsto \phi^{b}\phi^{-b}=1$ and $t_at_{a+b-1}^{-1}\mapsto (t_{a}t_{a+b-1}^{-1}\phi)^{b}\phi^{-b}=1$, we deduce $t_at_{a+b-1}^{-1}=1$, that is, $t_{a+b-1}=t_a$.

Moreover, from~\eqref{eq:saturday}, we have $t_a^{-1}=t_{a+b-1}^{-1}t_a^{\phi^{-1}}=t_a^{-1}t_a^{\phi^{-1}}$ and hence, from~\eqref{eq:13}, we get
\begin{align*}
t&=(t_{a-1}^{-1}\phi)^a\phi^{-a}=(t_{a}^{-1}t_a^{\phi^{-1}}\phi)^a\phi^{-a}=(t_{a}^{-1}\phi t_a)^a\phi^{-a}=t_a^{-1}\phi^at_a\phi^{-a}=t_a^{-1}t_a^{\phi^a}.
\end{align*}
Since this holds for every $t\in T$, this implies that the mapping $T\to T$ defined by $y\mapsto y^{-1}y^{\phi^a}$ is bijective. However, this contradicts Lemma~\ref{l:1}. 

Summing up, so far, we have shown the following result.
\begin{lemma}\label{l:14}
Each element of $L$ having at least two cycles (when written as the product of cycles having disjoint support) is $L$-conjugate to an element of $A$, that is, the elements in $L\setminus\bigcup_{\tau\in L}A^\tau$ consists of a unique cycle of length $\ell+1$.
\end{lemma}

From~\eqref{cl:1} and from Jordan's theorem, $L\setminus\bigcup_{\tau\in L}A^\tau\ne\emptyset$ and hence $L$ contains a permutation $\sigma$ which is a cycle of length $\ell+1$. Let $C=\langle\sigma\rangle$. From Lemma~\ref{l:14} and Proposition~\ref{p:1}, we deduce that
$$L=\bigcup_{\tau\in L}A^\tau\cup\bigcup_{\tau\in L}C^\tau.$$
Since $L$ is an epimorphic image of $G$ and since $\gamma(G/X)>2$ for every non-identity proper subgroup of $G$, we deduce $\gamma(L)>2$ and hence $A,C$ are not the components of a normal $2$-covering of $L$. As $A<L$ from~\eqref{cl:1}, we get $L=C$ is cyclic. Since $L$ acts primitively on $\{0,\ldots,\ell\}$, we get that $\ell+1$ is a prime number.

Summing up, $\ell+1$ is prime and $L$ is cyclic of order $\ell+1$, generated by $\sigma$. Without loss of generality, we may suppose that $$\sigma=(0\,1\,\ldots\,\ell).$$
Now, let $\phi\in  \mathrm{Aut}(T)$ such that $\phi\sigma\in K$. Replacing $\phi\sigma$ by a suitable power if necessary, we may suppose that $\phi$ has order a power of $\ell+1$. Suppose that $\ell+1$ divides $|T|$. Now, let $P$ be a Sylow $(\ell+1)$-subgroup of $\mathrm{Aut}(T)$ containing $\phi$ and let $Q=T\cap P$. Observe that $Q$ is a Sylow $(\ell+1)$-subgroup of $T$ and $Q\unlhd P$. As $\ell+1$ divides the cardinality of $T$, we get $Q\ne 1$ and hence ${\bf Z}(P)\cap Q\ne 1$. Let $z\in {\bf Z}(P)\cap Q$ having order $\ell+1$. From Lemma~\ref{l:13}, the mapping $y\mapsto (y\phi)^{\ell+1}\phi^{-(\ell+1)}$ is a bijection. Under this function, 
the element $z$ is mapped to $$(z^{-1}\phi)^{\ell+1}\phi^{-(\ell+1)}=z^{-(\ell+1)}\phi^{\ell+1}\phi^{-(\ell+1)}=1,$$
which is a contradiction because $z\ne 1$. Therefore,
\begin{align*}
\gcd(\ell+1,|T|)=1.
\end{align*}

Let $\phi\sigma\in K$ be an element having order a power of $\ell+1$, where $\phi\in\mathrm{Aut}(T)$. Now, for every $t\in T$, the element $(t,1,\ldots,1)\phi\sigma$ has a $G$-conjugate in $K$ and hence it has an $N$-conjugate in $K$, because $G=NK$. Therefore there exist $t_1,\ldots,t_\ell\in T$ such that by setting $n=(1,t_1,\ldots,t_\ell)$ we have
\begin{align*}
((t,1,\ldots,1)\phi\sigma)^n&=(tt_1^{\phi^{-1}},t_1^{-1}t_2^{\phi^{-1}},t_2^{-1}t_3^{\phi^{-1}},\ldots,t_{\ell-1}^{-1}t_\ell^{\phi^{-1}},t_\ell^{-1})\phi\sigma\in K.
\end{align*}
This implies
$$tt_1^{\phi^{-1}}=t_1^{-1}t_2^{\phi^{-1}}=t_2^{-1}t_3^{\phi^{-1}}=\cdots=t_{\ell-1}^{-1}t_\ell^{\phi^{-1}}=t_\ell^{-1}.$$
From this we deduce that, for each $i\in \{1,\ldots,\ell-1\}$,
$$t_i=t_\ell^{\phi^{-(\ell-i)}}t_\ell^{-(\ell-i-1)}\cdots t_\ell^{\phi^{-1}}t_\ell$$
and $$t=t_\ell^{-1}(t_\ell^{-1})^{\phi^{-1}}\cdots (t_\ell^{-1})^{\phi^{-\ell}}.$$
This implies that the mapping $\varphi:T\to T$ defined by $y\mapsto yy^{\phi^{-1}}\cdots y^{\phi^{-\ell}}$ is surjective, and hence bijective. Let $z\in T$ and set $y=z^{-1}z^{\phi^{-1}}$. We get
$$\varphi(y)=z^{-1}z^{\phi^{-1}}(z^{-1}z^{\phi^{-1}})^{\phi^{-1}}\cdots (z^{-1}z^{\phi^{-1}})^{\phi^{-\ell}}=z^{-1}z^{\phi^{-(\ell+1)}}.$$
This shows that, for every $z\in{\bf C}_T(\phi^{\ell+1})$, $\varphi$ maps $y=z^{-1}z^{\phi^{-1}}$ to $$\varphi(y)=z^{-1}z^{\phi^{-(\ell+1)}}=1.$$ As $\varphi$ is injective, we deduce that
 for every $z\in{\bf C}_T(\phi^{\ell+1})$, we have $z^{-1}z^{\phi^{-1}}=1$, that is, $z^{\phi}=z$. This yields ${\bf C}_T(\phi^{\ell+1})={\bf C}_T(\phi)$. From Lemma~\ref{l:new}, $\phi=1$. We have shown that the only non-identity element of $K$ having order a power of $\ell+1$ are the elements in $\langle \sigma\rangle$.

This shows that $$K=U\times \langle(0\,1\,\cdots\,\ell)\rangle,$$
where $T\le U\le\mathrm{Aut}(T)$ and $\gcd(|U|,\ell+1)=1$. 
\end{proof}
\end{document}